\documentclass[12pt]{amsart}
\usepackage{amscd,amssymb,amsthm,verbatim,tikz-cd}
\newcommand{\C}{{\mathbb C}}

\newcommand{\ch}{\operatorname{ch}}

\newcommand{\dvol}{\operatorname{dvol}}
\newcommand{\End}{\operatorname{End}}

\newcommand{\GL}{\operatorname{GL}}

\newcommand{\HH}{\operatorname{H}}

\newcommand{\Ker}{\operatorname{Ker}}

\newcommand{\R}{{\mathbb R}}

\newcommand{\tr}{\operatorname{tr}}
\newcommand{\Td}{\operatorname{Td}}
\newcommand{\Tr}{\operatorname{Tr}}

\newcommand{\Z}{{\mathbb Z}}

\numberwithin{equation}{section}
\theoremstyle{plain}

\newtheorem{definition}[equation]{Definition}

\newtheorem{theorem}[equation]{Theorem}

\theoremstyle{remark}

\usepackage{graphicx}
\input{amssym.def}
\input{amssym}
\input{epsf}
\usepackage{a4wide}

\begin{document}
\title{The Ray-Singer Torsion}

\author{John Lott}
\address{Department of Mathematics\\
University of California, Berkeley\\
Berkeley, CA  94720-3840\\
USA} \email{lott@berkeley.edu}

\thanks{}
%\date{September 17, 2023}
\begin{abstract}
In 1971, Ray and Singer proposed an 
analytic equivalent of a classical 
topological invariant, the $R$-torsion.
This Ray-Singer torsion has had many
ramifications in mathematics and physics.
I will describe the background, the Ray-Singer papers and some
subsequent work.
\end{abstract}
\maketitle

\tableofcontents

\noindent

\section{Introduction} \label{sect1}

As a bridge between two worlds, it's always interesting to find analytic equivalents for topological invariants.  An outstanding example is the Hodge theorem, relating the real cohomology of a smooth compact manifold to the harmonic differential forms, after equipping the manifold with a Riemannian metric.

In a 1971 paper \cite{Ray-Singer (1971)}, Daniel Ray and Isadore Singer of MIT proposed an analytic equivalent of a classical topological invariant called the $R$-torsion.  In so doing, they introduced a remarkable notion of a regularized determinant of a Laplace-type operator.  
The Ray-Singer work turned out to have contact with many areas of mathematics and physics, in ways that would have been hard for Ray and Singer to predict. In a second paper in 1973, they extended their methods to the holomorphic setting \cite{Ray-Singer (1973)}.

After the Ray-Singer papers appeared, it was probably not clear how the Ray-Singer torsion fit into the wider framework
of geometry and topology.  This was clarified in later decades by looking at families
of manifolds, rather than individual manifolds.

In Section \ref{sect2}, I'll review some of the work that lead up to the Ray-Singer papers.  Section \ref{sect3} describes the content of the Ray-Singer papers.  Section \ref{sect4} has selected further developments.
To give a coherent narrative, I haven't tried to be comprehensive.
To restrict the length of this article, most of the work described is from the $20^{th}$ century.

There are excellent expositions of the Ray-Singer work, such as Werner M\"uller's article in the Notices
\cite{Muller (2022)}.  I thank Jeff Cheeger and Dan Freed for their comments on an earlier version of this article.

\section{Topological precedents} \label{sect2}

There were both topological and analytical results that inspired the Ray-Singer work.  Although I was a thesis student of Singer, I never asked him about the motivations for the Ray-Singer papers, so some of what follows are my surmises.

Section \ref{sect2.1} has a description of lens spaces.  Section \ref{sect2.2} has the definition of $R$-torsion
and Section \ref{sect2.3} gives some of its properties.  Some applications of the $R$-torsion are given
in Section \ref{sect2.4}, namely the homeomorphism classification of lens spaces and the disproof of
the Hauptvermutung for simplicial complexes.  Section \ref{sect2.5} mentions the role of Arnold Shapiro.

A reference for the material in this section is
\cite{Milnor (1966)}.

\subsection{Lens spaces} \label{sect2.1}

A basic problem in topology is to understand the homeomorphism types of manifolds.
This is already challenging in three dimensions.
There's a class of three dimensional manifolds called {\it lens spaces} which, although easy to define, have interesting topological properties.

Let $p > q$ be relatively prime positive integers.  Writing
$S^3 = \{(z_1, z_2) \in \C^2 \: : \: |z_1|^2 + |z_2|^2 = 1 \}$,
there's a free $\Z_p$-action on $S^3$ whose generator sends
$(z_1, z_2)$ to $\left( e^{\frac{2 \pi i}{p}} z_1, 
e^{\frac{2 \pi iq}{p}} z_2 \right)$.  The lens space $L(p,q)$ is the quotient of $S^3$ by this $\Z_p$-action. The word ``lens'' comes from a way of picturing a fundamental domain for the action \cite[Example 1.4.6]{Thurston (1997)}. There are analogous higher dimensional lens spaces in which $\Z_p$ acts on an odd dimensional sphere.

When is $L(p,q)$ homeomorphic to $L(p^\prime, q^\prime)$?  Considering fundamental groups, $p$ must be equal to $p^\prime$. One gets some further information from homotopy theory.  However, it turns out that $L(7,1)$ is homotopy equivalent to $L(7,2)$, so standard algebraic topology won't help to decide whether they are homeomorphic.

\subsection{$R$-torsion} \label{sect2.2}
The $R$-torsion was developed by Kurt Reidemeister \cite{Reidemeister (1935)}, Wolfgang Franz \cite{Franz (1935)} and Georges de Rham \cite{de Rham (1936)} to understand when lens spaces
are combinatorially equivalent.
Franz did his Habilitation degree under the supervision of
Reidemeister.\footnote{Reidemeister was removed from his professorship at the University of K\"onigsberg in 1933, as retaliation for his anti-Nazi statements. He learned of his dismissal by reading about it in the local newspaper
\cite{BBF (1972)}.  He got a position at the University of Marburg in 1934, where he remained until 1955. 
During the war Franz worked in the Wehrmacht's codebreaking group. He held a position at the Goethe University Frankfurt between 1946 and 1974.} Reidemeister's work was about three dimensional manifolds and had the aim of classifying the lens spaces up to combinatorial equivalence.  De Rham and Franz extended Reidemeister's work to higher dimension. Franz dealt with arbitrary finite simplicial complexes.   

To establish some notation, let $V$ be a finite dimensional real vector space. If ${\bf v} = (v_1, \ldots v_n)$ and
${\bf w} = (w_1, \ldots w_n)$ are two bases of $V$ then 
we write $[{\bf w}/{\bf v}]$ for $| \det T|$, where $T$ is
the change-of-basis matrix from ${\bf v}$ to ${\bf w}$,
i.e. $w_i = \sum_j T_{ij} v_j$.

Now let $C$ be a finite chain complex 
\begin{equation} \label{2.1}
C_N \stackrel{\partial}{\longrightarrow} C_{N-1}
\stackrel{\partial}{\longrightarrow} \ldots 
C_1 \stackrel{\partial}{\longrightarrow} C_{0}
\end{equation}
of finite dimensional real vector spaces.
As usual, we write $Z_q \subset C_q$ and $B_q \subset C_q$
for the kernel and image of $\partial$, respectively, and 
put $H_q = Z_q/B_q$, so that we have short exact
sequences 
\begin{equation} \label{2.2}
0 \longrightarrow Z_q \longrightarrow C_q \stackrel{\partial}{\longrightarrow} B_{q-1} \longrightarrow 0
\end{equation}
and 
\begin{equation} \label{2.3}
0 \longrightarrow B_q \longrightarrow Z_q \longrightarrow H_q \longrightarrow 0.
\end{equation}

Suppose that ${\bf c}_q$ is a preferred basis for $C_q$ and
${\bf h}_q$ is a preferred basis for $H_q$.   Choose an
auxiliary
basis ${\bf b}_q$ for $B_q$. 

Choose elements $\tilde{\bf b}_{q-1} \subset C_q$ so that $\partial \tilde{\bf b}_{q-1} = {\bf b}_{q-1}$ and choose elements
$\tilde{\bf h}_q \subset Z_q$ that project to ${\bf h}_q$. Then
$({\bf b}_q, \tilde{\bf h}_q, \tilde{\bf b}_{q-1})$ is a basis
for $C_q$. Since $[({\bf b}_q, \tilde{\bf h}_q, \tilde{\bf b}_{q-1})/{\bf c}_q]$ only depends on 
${\bf b}_q$, ${\bf h}_q$, ${\bf b}_{q-1}$ and ${\bf c}_q$, we write it as
$[({\bf b}_q, {\bf h}_q, {\bf b}_{q-1})/{\bf c}_q]$.

\begin{definition} \label{2.4}
The torsion of $C$ is the positive real number $\tau(C)$ given by
\begin{equation} \label{2.5}
\log \tau(C) = \sum_{q=0}^N (-1)^q \log [({\bf b}_q, {\bf h}_q, {\bf b}_{q-1})/{\bf c}_q].
\end{equation}
\end{definition}

It's not hard to see that $\tau(C)$ is independent of the choices of the ${\bf b}_q$ bases. And it just depends on the
${\bf c}_q$ and ${\bf h}_q$ bases through their induced volume forms on $C_q$ and $H_q$, respectively.

If $K$ is a connected finite simplicial complex then its real chain groups have preferred bases and we can talk about the torsion of
the chain complex, provided that we are given a preferred basis of the real homology groups.  In order to compare two different simplicial complexes
using numerical invariants, it turns out to be useful to have {\em acyclic} complexes, i.e. ones with vanishing homology, so that no choice of 
basis is needed.  This can often be achieved by {\em local systems}
on $K$.  For us, a local system can be thought of as a flat Euclidean vector bundle on $K$ or, more algebraically, as arising from a
homomorphism $\rho : \pi_1(K, k_0) \rightarrow O(n)$. Letting $\widetilde{K}$ denote the universal cover of $K$, the twisted chain groups are given by
$C_q(K, \rho) = C_q(\widetilde{K}) \otimes_{\R \pi_1(K,k_0)} \R^n$,
where an element $\sum_g a_g g$ of $\R \pi_1(K,k_0)$ acts on 
$v \in \R^n$ by $(\sum_g a_g g) \cdot v = \sum_g a_g \: \rho(g) v$.

Choosing a fundamental domain of $K$ in $\widetilde{K}$, the
natural basis of $C_q({K})$ and the standard basis of
$\R^n$ combine to give a preferred basis of $C_q(K, \rho)$.

\begin{definition} \label{2.6}
Given the finite simplicial complex $K$ and a representation 
$\rho$ so that $C(K, \rho)$ has vanishing homology,
the $R$-torsion is given by $\tau_K(\rho) = \tau(C(K, \rho))$.
\end{definition}

One can show that $\tau_K(\rho)$ is independent of the choice of fundamental domain in $\widetilde{K}$.  

\subsection{Properties of the $R$-torsion} \label{sect2.3}
The main point of the
construction  of $\tau_K(\rho)$ is that it is
{\em invariant under subdivison}. In particular, if two triangulated manifolds can be 
subdivided to become combinatorially equivalent then they have the same
values for the $R$-torsion. 
It follows that the $R$-torsion is a PL (piecewise linear)
invariant of PL manifolds.  In particular, it is a diffeomorphism invariant of
smooth manifolds. (Some of the early literature doesn't distinguish clearly 
between PL homeomorphism and topological homeomorphism.)

Now $\tau_K$ is not a homotopy invariant of $K$.
However, it is invariant under a more restricted notion of homotopy
equivalence, called {\em simple homotopy equivalence}
\cite{Cohen (1973)}.  

As to whether two homeomorphic complexes have the same torsion, this will be true if the complexes are
necessarily simple homotopy equivalent. Such was proven for
PL manifolds by Kirby and Siebenmann \cite{Kirby-Siebenmann (1977)} and more generally for
CW complexes by Chapman \cite{Chapman (1974)}. Chapman's proof involved a foray into ``manifolds'' locally modelled on Hilbert cubes.

In some ways, the $R$-torsion is a cousin of the Euler characteristic.  For example, the Euler characteristic of an 
odd dimensional closed (= compact boundaryless) manifold vanishes, while the $R$-torsion of
an even dimensional (oriented) closed manifold vanishes.
There is a product formula: If $K$ and $K^\prime$ are finite
complexes, and $\rho : \pi_1(K^\prime, k_0^\prime) \rightarrow
O(n)$ is a homomorphism, let $\widehat{\rho}$ be the composite
homomorphism $\pi_1(K \times K^\prime, (k_0,k_0^\prime)) 
\rightarrow \pi_1(K^\prime, k_0^\prime) \rightarrow
O(n)$. Assuming that $C(K^\prime, \rho)$ is acyclic, one finds that
\begin{equation} \label{2.7}
T_{K \times K^\prime}(\widehat{\rho}) = 
\chi(K) T_{K^\prime}({\rho}).
\end{equation}
The relation between $R$-torsion and Euler characteristic will be clarified in Section \ref{sect5.4}.

\subsection{Applications of the $R$-torsion} \label{sect2.4}

Returning to lens spaces, we can triangulate $L(7,1)$ and $L(7,2)$.
Running through the homomorphisms from
$\Z_7$ to $O(2)$, one finds that the possible numerical values of 
the $R$-torsion differ for $L(7,1)$ and $L(7,2)$. Hence they cannot
be homeomorphic. (Without using \cite{Chapman (1974)} or \cite{Kirby-Siebenmann (1977)}, the $R$-torsion computations imply that
$L(7,1)$ and $L(7,2)$ are not PL homeomorphic, and in three
dimensions topological manifolds have unique PL structures
\cite{Moise (1952)}.)

More generally, one sees that $L(p,q)$ is homeomorphic to
$L(p, q^\prime)$ if $q^\prime \equiv \pm q \mod p$ (coming from the
involution $z_2 \rightarrow \overline{z_2}$) or
$\pm qq^\prime \equiv 1 \mod p$ (coming from the additional involution $(z_1, z_2) \rightarrow (z_2, z_1)$).
The $R$-torsion shows that this is the only way that $L(p,q)$ can be {\em homeomorphic} to
$L(p, q^\prime)$.  In contrast, $L(p,q)$ is {\em homotopy equivalent} to
$L(p, q^\prime)$ if and only if $\pm q q^\prime$ is a quadratic residue mod $p$.

There are similar statements for higher dimensional lens spaces.
More generally, consider {\em spherical space forms}, meaning quotients
$M = S^r/\Gamma$ where $\Gamma$ is a finite subgroup of $O(r+1)$ that
acts freely on $S^r$. De Rham showed that spherical space forms of a given dimension can be
classified up to isometry by their fundamental groups and $R$-torsions \cite{de Rham (1950)}. As a consequence, two spherical space forms are homeomorphic if and only if they are isometric.

A classical application of the $R$-torsion was to disprove the
{\em Hauptvermutung} (or ``main conjecture'') for simplical complexes.
The motivation went back to the beginnings of
homology theory. The homology groups of a finite simplicial complex
were first defined combinatorially using simplicial homology.  While
simplicial homology had many nice features, it was not at all clear
whether {\em homeomorphic} simplicial complexes had isomorphic simplicial
homology.  One approach to this was to try to prove the 
Hauptvermutung, saying that 
homeomorphic simplicial complexes have combinatorially equivalent subdivisions.
Using the subdivision invariance of simplicial homology, one would then conclude that simplicial homology is homeomorphism invariant.

The homeomorphism invariance of simplicial homology was
proven instead using the simplicial approximation theorem, but the 
validity of the Hauptvermutung stayed open until 1961
when John Milnor gave a counterexample
\cite{Milnor (1961)}.  Let
$\sigma^r$ denote the $r$-simplex.  For $j \in \{1,2\}$, let $L_j$ be a
triangulation of $L(7,j)$.  Consider the product $L_j \times \sigma^r$
and let $X_j$ denote the result of coning off
its boundary $L_j \times \partial \sigma^r$.  If $r \ge 3$
then $X_1$ is homeomorphic to $X_2$. The reason is that from an
argument of Barry Mazur, $L_1 \times \R^r$ is homeomorphic to $L_2 \times \R^r$,
and $X_j$ is homeomorphic to the one point compactification of $L_j \times \R^r$. 
On the combinatorial side, $X_1$ and $X_2$ are simply connected, so one
can't use the $R$-torsion directly. Milnor instead
used a variant of the $R$-torsion. Let $Y_j$ be the one point compactification of the universal cover of
$L_j \times \R^r$, with $y_{j,0}$ being the added point.  It inherits a cellular structure on which
$\Z_7$ acts, freely off of $y_{j,0}$, with quotient $X_j$.  Using a representation
$\rho : \Z_7 \rightarrow O(2)$,
the torsion of the relative chain complex 
$C_*(Y_j, y_{j,0}) \otimes_{\R \Z_7} \R^2$ and its subdivision invariance, Milnor showed that
$X_1$ and $X_2$ do not have combinatorially equivalent subdivisions.

\subsection{Arnold Shapiro} \label{sect2.5}

In the first Ray-Singer paper, one sees, ``We raise the question as to how to describe this
manifold invariant in analytic terms.  Arnold Shapiro once suggested that there might be a
formula for the torsion in terms of the Laplacian $\Delta$ acting on differential forms.''

Arnold Shapiro was a topologist who was born in 1921 and died in 1962, 
nine years before the first Ray-Singer paper appeared.  Shapiro
was a professor at Brandeis and was apparently the first person to come up with
an explicit way to turn the sphere inside out. He is
 known for his paper with Atiyah and Bott on Clifford modules and
K-theory, which appeared two years after his death.  Referring to
the period 1955-1957,
Raoul Bott wrote, ``During that time, and largely at Princeton, I met Serre, Thom, Hirzebruch, Atiyah, Singer, Milnor, Borel, Harish-Chandra, James, Adams,... I could go on and on. But these people, together with Kodaira and Spencer, and my more or less `personal remedial tutor', Arnold Shapiro, were the ones I had the most mathematical contact with."

Regarding Shapiro's suggestion, there is a way to write the
$R$-torsion in terms of {\em combinatorial} Laplacians.  In reference to
the chain complex (\ref{2.1}), the preferred basis of $C_q$ defines an inner
product on $C_q$ for which the basis elements are orthonormal.  Let $\partial^* : C_q \rightarrow C_{q+1}$
be the adjoint operator to $\partial$.  Define the combinatorial
Laplacian $\triangle^{(c)}_q : C_q \rightarrow C_q$ by
$\triangle^{(c)} = - \left( \partial^* \partial + \partial \partial^*
\right)$, a nonpositive self-adjoint operator. If the chain complex $C$ is acyclic then each $\triangle^{(c)}_q$ is invertible and
\cite[Proposition 1.7]{Ray-Singer (1971)} states that
\begin{equation} \label{2.8}
\log \tau(C) = \frac12 \sum_{q=0}^N (-1)^{q+1} q \log 
\det( - \triangle^{(c)}_q).
\end{equation}

\section{Analytical precedents} \label{sect3}

This section describes some of the analytical work leading up to the
Ray-Singer papers.  In Section \ref{sect3.1}, I recall the relation between the
Riemann zeta function and heat conduction on a circle.  Section \ref{sect3.2}
describes the work of Minakshisundaram and Pleijel on defining the zeta
function of the Laplacian on a compact Riemannian manifold.  Section
\ref{sect3.3} is about the subsequent work by McKean and Singer.
A reference for the material in this section is
\cite[Chapter 2]{Berline-Getzler-Vergne (2004)}.

Following Shapiro's suggestion to write the $R$-torsion in terms of
differential form Laplacians, and in view of (\ref{2.8}), 
Ray and Singer faced the task of making sense of
the determinant of a self-adjoint operator on a Hilbert space.  The relevant
operators had discrete spectrum but the product of the eigenvalues was
not convergent.  Instead it was typically a product such as
$1 \cdot 2 \cdot 3 \cdot \ldots$.

There is a well known way to take the {\em sum} of such numbers, giving
the semiserious formula
$1 + 2 + 3 + \ldots = - \frac{1}{12}$.
The meaning of this formula is in terms of the Riemann zeta function
\begin{equation} \label{3.1}
\zeta(s) = 1 + \frac{1}{2^s} + \frac{1}{3^s} + \frac{1}{4^s} + \ldots.
\end{equation}
Formally, $1 + 2 + 3 + \ldots =  \zeta(-1)$.  On the other hand,
we know that the zeta function can be meromorphically continued from
$\Re(s) > 1$ to the complex plane and its value at $s = -1$ is rigorously $- \frac{1}{12}$.

Ray and Singer showed how to give a meaning to determinants of Laplacians on 
manifolds, using analytic continuation.  

\subsection{Riemann zeta function and heat conduction} \label{sect3.1}

Let us first recall the relationship
between the Riemann zeta function and
heat conduction on a circle. 

Let $S^1(L)$ denote a circle of length $L$.
The heat equation on the circle, with initial condition $u_0$,
is 
\begin{equation} \label{3.2}
\frac{\partial u}{\partial t} = \frac{\partial^2 u}{\partial x^2} , \: \: \: \: \: \: u(0,x) = u_0(x).
\end{equation}
Here $u_0$ is a function on $S^1(L)$ or, equivalently, an $L$-periodic
function on $\R$.
The  time-$t$ solution can be written as $u_t(x) = \int_{S^1(L)} K(t,x,y) \:
u_0(y) \: dy$ where $K(t,x,y)$ is the heat kernel or, in operator terms, as $u_t =  e^{t \partial_x^2} u_0$.  
Just as the trace of a matrix can be written as a sum of diagonal entries, we can write
\begin{equation} \label{3.3}
\Tr \left( e^{t \partial_x^2} \right) = \int_{S^1(L)} K(t,x,x) \: dx.
\end{equation}

The eigenvalues of
$\partial_x^2$ are zero, with multiplicity one, and 
$\left\{- \left( \frac{2 \pi j}{L} \right)^2 \right\}_{j=1}^\infty$,
each with multiplicity two. Hence
\begin{equation} \label{3.4}
\Tr \left( e^{t \partial_x^2} \right) =
1 + 2 \sum_{j=1}^\infty e^{- t \left( \frac{2 \pi j}{L} \right)^2}.
\end{equation}
Using the formula $\lambda^{-s} = \frac{1}{\Gamma(s)} \int_0^\infty t^{s-1} e^{- \lambda t} dt$ for positive $s$ and $\lambda$, we 
find that
\begin{equation} \label{3.5}
2 \left( \frac{L}{2\pi} \right)^{2s} \zeta(2s) = \frac{1}{\Gamma(s)} \int_0^\infty t^{s-1} \left(
\Tr \left( e^{t \partial_x^2} \right) - 1 \right) dt
\end{equation}
provided that $\Re(s) > \frac12$. 

We could describe the meromorphic
continuation of (\ref{3.5}) using what is known about the Riemann zeta function but to give a direct description, 
we choose $\epsilon > 0$ and write the right-hand side as a sum of two terms, the first coming from the $t$-integration
between $0$ and $\epsilon$, and the second coming from the $t$-integration between $\epsilon$ and $\infty$.  
Since $\Tr \left( e^{t \partial_x^2} \right) - 1$ decays exponentially fast in $t$, the
second term is analytic in $s$. To handle the first term, there is an explicit formula
\begin{equation} \label{3.6}
K(t,x,y) = \frac{1}{\sqrt{4 \pi t}} \sum_{k \in \Z} 
e^{- \frac{|x-y-kL|^2}{4t}}
\end{equation}
for $x, y \in [0, L)$, 
coming from the heat kernel $\frac{1}{\sqrt{4 \pi t}}  
e^{- \frac{|x-y|^2}{4t}}$ on $\R$ and thinking of $x$ as receiving heat
from sources $\{y+kL\}_{k \in \Z}$ in $\R$.
Then the first term is
\begin{align} \label{3.7}
& \frac{1}{\Gamma(s)} \int_0^\epsilon t^{s-1} \left(
\Tr \left( e^{t \partial_x^2} \right) - 1 \right) dt =
 \frac{1}{\Gamma(s)} \int_0^\epsilon t^{s-1} \left(
\int_{S^1(L)} K(t,x,x) \: dx - 1 \right) dt = \\
& \frac{1}{\Gamma(s)} \int_0^\epsilon t^{s-1} \left(
\frac{1}{\sqrt{4 \pi t}} L \sum_{k \in \Z}
e^{- \frac{L^2 k^2}{4t}}
 - 1 \right) dt. \notag
\end{align}
The only possible singularities come from the $k=0$ term, i.e.
$\frac{1}{\Gamma(s)} \int_0^\epsilon t^{s-1}
\frac{L}{\sqrt{4 \pi t}} \: dt$, which for large $s$ equals $\frac{L}{\sqrt{4 \pi}} \frac{\epsilon^{s-\frac12}}{\Gamma(s)}
\frac{1}{s-\frac12}$. Hence the only singularity is a simple pole at $s = \frac12$ with residue $\frac{L}{2\pi}$.

There is a similar story when the circle is replaced with a flat torus of dimension $N$.  The Riemann zeta function is
replaced by the Epstein zeta function.  The analogous expression to (\ref{3.5}) acquires a simple pole at $\frac{N}{2}$ with a
residue equal to $\frac{1}{(4\pi)^{\frac{N}{2}} \Gamma(\frac{N}{2})}$ times the volume of the torus.

\subsection{Minakshisundaram-Pleijel} \label{sect3.2}

Going from circles to manifolds,
Subbaramiah Minakshisundaram, from India, and \r{A}ke Pleijel, from Sweden, 
defined the zeta function of the Laplacian on a compact Riemannian manifold
\cite{Minakshisundaram-Pleijel (1949)}.
The collaboration happened when they were
visiting the Institute for Advanced Study during the 1947-1948 academic year.
They had each worked on related problems for domains in the plane but neither of them had used Riemannian manifolds before. 

If $\triangle$
denotes the (nonpositive) Laplace operator then heat conduction on a closed Riemannian
manifold $M$ satisfies the equation
\begin{equation} \label{3.8}
\frac{\partial u}{\partial t} = \triangle u , \: \: \: \: \: \: u(0,x) = u_0(x).
\end{equation}
The solution can be written $u_t(x) = \int_{M} K(t,x,y) \:
u_0(y) \: \dvol(y)$, where $K(t,x,y)$ is the heat kernel, or in operator terms as $u(t) =  e^{t \triangle} u_0$. The method of
Minakshisundaram-Pleijel was to first write an approximate
solution to $K(t,x,y)$ in normal coordinates around a point
$y \in M$.  Their approximate solution or ``parametrix'' was of the form
\begin{equation} \label{3.9}
H(t,x,y) = (4 \pi t)^{- \frac{N}{2}} e^{- \frac{r^2}{4t}} \left( U_0 + U_1 t + \ldots +
U_n t^n \right).
\end{equation}
Here $N = \dim(M)$, $r$ is the distance from $x$ to $y$, 
each $U_j$ is a function of $x$ and $y$, and $n$ is a parameter.
The $U_j$'s were computed recursively by
requiring that $H$ satisfy the PDE (\ref{3.8}) to leading order. This gave a formula for
$U_j$ in terms of $U_{j-1}$; the starting point was $U_0 = 1$.

Minakshisundaram and Pleijel were able to estimate the error when approximating $K$ by $H$. 
(More precisely, they passed to Green's functions.)  They then considered
the expression
\begin{equation} \label{3.10}
\zeta_{x,y}(s) = \frac{1}{\Gamma(s)} \int_0^\infty t^{s-1} \left( K(t,x,y) - 1 \right) \: dt.
\end{equation}
If $\{\lambda_j\}_{j=1}^\infty$ are the nonzero eigenvalues of $\triangle$ then
the corresponding zeta function is defined to be
\begin{equation} \label{3.11}
\zeta_{\triangle}(s) = \sum_j (- \lambda_j)^{-s} = \int_M \zeta_{x,x}(s) \: \dvol(x).
    \end{equation}
        The main result of Minakshisundaram and Pleijel was the
        following.

        \begin{theorem} \label{3.12} \cite{Minakshisundaram-Pleijel (1949)}
            If $x \neq y$ then $\zeta_{x,y}(s)$ extends to an
            analytic function of $s$.  If $x = y$ then
           $\zeta_{x,x}(s)$ extends to a meromorphic function of $s$.
           If $N$ is odd then $\zeta_{x,x}(s)$ has simple poles at
           $\frac{N}{2} - j$ for $j = 0, 1, 2, \ldots$, while if $N$ is
           even then $\zeta_{x,x}(s)$ has simple poles at 
           $\frac{N}{2}, \frac{N}{2} - 1, \ldots, 1$.   
           If $N$ is odd then $\zeta_{x,x}(s)$ vanishes at
           nonpositive integers.
           
           There is a similar statement for $\zeta_{\triangle}(s)$.
        \end{theorem}

        The statements about the locations of the poles of 
        $\zeta_{x,x}(s)$ and the values at nonpositive integers can
        be seen 
        by plugging the parametrix (\ref{3.9}) into (\ref{3.10})
        and integrating from $0$ to $\epsilon$. 
        If $N$ is even then the values of $\zeta_{x,x}(s)$ at
        nonpositive integers can be computed this way.
        In particular, they are expressions in the Riemannian 
metric and its derivatives, up to a certain order, at $x$.

\subsection{McKean-Singer} \label{sect3.3}

A 1967 paper by Henry McKean and Singer extended the
Minakshisundaram-Pleijel work in several ways \cite{McKean-Singer (1967)}.
First, the coefficients $U_1(x,x)$ and $U_2(x,x)$ in (\ref{3.9}) were
computed. They found that $U_1(x,x) = \frac{R(x)}{6}$, where $R$ is
the scalar curvature, and $U_2(x,x)$ is the sum of a quadratic
expression in the curvature tensor and a multiple of $\triangle R$.

Second, McKean and Singer considered Laplacians on manifolds with
boundary, with Dirichlet or Neumann boundary conditions. They 
constructed a parametrix and showed that there is again an
asymptotic expansion
\begin{equation} \label{3.13}
\Tr \left( e^{t \triangle} \right) = \int_M K(t,x,x) \: \dvol(x) \sim (4 \pi t)^{- \frac{N}{2}} 
(c_0 + c_1 t^{\frac12} + c_2 t + \ldots)
\end{equation} 
for small $t$, but now with
half-integer powers.

Third, and what is most relevant for the Ray-Singer papers, they discussed 
such asymptotic expansions when the function Laplacian is replaced by the
Hodge Laplacian on differential forms.  Their motivation for this discussion
came from the possibility of proving the Chern-Gauss-Bonnet theorem using
heat equation methods, a possibility that was later realized.

\section{The Ray-Singer papers} \label{sect4}

The Ray-Singer papers were joint works between Daniel Ray
and Isadore Singer.  Ray was a faculty member at MIT from 1957 to 1979.
His
specialties were stochastic processes and spectral problems. 
I had Ray as a teacher for undergraduate analysis.

Section \ref{sect4.1} describes the first Ray-Singer paper and Section \ref{sect4.2}
describes the second Ray-Singer paper.

\subsection{The first Ray-Singer paper} \label{sect4.1}

Recall equation (\ref{2.8}), giving the $R$-torsion in terms of
combinatorial Laplacians, and Shapiro's suggestion that there
may be a formula for the torsion in terms of the Laplacian acting
on differential forms. 

Whereas the determinants in (\ref{2.8}) are of finite dimensional
operators, it is not clear what the determinant of an
(infinite dimensional) differential
form Laplacian should mean.  It was for this purpose that Ray and
Singer introduced {\em zeta function regularized determinants}.
To motivate this notion, consider the calculus formula
\begin{equation} \label{4.1}
\log \lambda = - \frac{d}{ds} \Big|_{s=0} e^{-s \log(\lambda)} = - \frac{d}{ds} \Big|_{s=0} \lambda^{-s}
\end{equation}
for $\lambda$ positive. Extending this to matrices,
if $M$ is a positive definite square matrix then
using the spectral theorem,
$\Tr \left( M^{-s} \right)$ is an analytic function of $s$ and
\begin{equation} \label{4.2}
\log \det(M) = - \frac{d}{ds} \Big|_{s=0} \Tr \left( M^{-s} \right).
\end{equation}
The idea of Ray and Singer was to use (\ref{4.2}) to {\em define}
the determinant of suitable infinite dimensional operators.

To specify the relevant operators, 
let $W$ be a connected closed orientable Riemannian manifold and 
let $E$ be a flat orthogonal vector bundle on $W$ or, equivalently,
a homomorphism $\rho : \pi_1(W, w_0) \rightarrow O(n)$.
Let $\Omega^q(W, E)$ denote the smooth $q$-forms on $W$ with value in $E$.
The Riemannian metric on $W$, along with the standard inner product on $\R^n$, 
gives an $L^2$-inner product on $\Omega^q(W, E)$. The exterior derivative
$d : \Omega^*(W, E) \rightarrow \Omega^{*+1}(W, E)$ has a formal adjoint
$\delta : \Omega^*(W, E) \rightarrow \Omega^{*-1}(W, E)$. The differential
form Laplacian is defined to be $\triangle = - (\delta d + d \delta)$.  Let
$\triangle_q$ denote the restriction of $\triangle$ to $\Omega^q(W, E)$.

Let us initially suppose that $\triangle_q$ is negative definite for all $q$.
By the Hodge theorem, this is equivalent to saying that 
$\HH^q(W, E) = 0$ for all $q$. Putting $\zeta_{q,\rho}(s) = 
\Tr \left( 
(- \triangle_q)^{-s} \right)$,
we can {\em define} $\det (\triangle_q)$ by 
\begin{equation} \label{4.3}
\log \det(\triangle_q) = - \frac{d}{ds} \Big|_{s=0} \zeta_{q,\rho}(s).
\end{equation}
The key point is that $\zeta_{q,\rho}(s)$ is analytic near $s = 0$, so
the definition makes sense. If we think of descending from
$s$ large, where $\zeta_{q,\rho}(s)$ makes conventional sense, 
then we encounter poles in $\zeta_{q,\rho}(s)$ at
$s = \frac{N}{2}, \frac{N}{2} - 1$, etc.  So 
$\log \det(\triangle_q)$ can only be computed after traversing
all of these poles.

\begin{definition} \label{4.4}
Suppose that the flat vector bundle $E$ is acyclic, i.e. that
$\HH^*(W, E) = 0$. The analytic torsion is the positive real number
$T_W(\rho)$ such that
\begin{equation} \label{4.5}
\log T_W(\rho
) = \frac12 \sum_{q=0}^N (-1)^q q \zeta^\prime_{q,\rho}(0).
    \end{equation}
\end{definition}

Note the similarity with (\ref{2.8}).
Using Hodge duality, one finds that $T_W(\rho) = 1$ if $W$ is even dimensional.
The main result of the Ray-Singer
paper, that $T_W(\rho)$ is a diffeomorphism invariant of $W$, follows
from the next theorem.

\begin{theorem} \label{4.6} \cite{Ray-Singer (1971)}
$T_W(\rho)$ is independent of the Riemannian metric on $W$.
\end{theorem}
\begin{proof}
We may assume that $W$ is odd dimensional.  Suppose that $g_0$ and $g_1$ are two
Riemannian metrics on $W$.  Putting $g(u) = u g_1 + (1-u) g_0$ and computing
the analytic torsion with respect to $g(u)$, it suffices to
show that $\frac{d}{du} \log T_W(\rho) = 0$.

In terms of the Hodge duality operator $\star$, one can write
$\delta = \pm \star^{-1} \circ d \circ \star$. Putting $\alpha = \star^{-1}
\frac{d \star}{du}$, one has
$\frac{d\delta}{du} = [\delta, \alpha]$. Then
$\frac{d\triangle}{du} = -[\delta, \alpha]d - d [\delta, \alpha]$.

Now $\frac{d}{du} \Tr (- \triangle_q)^{-s} = 
s \Tr \left( \frac{d \triangle_q}{du} (- \triangle_q)^{-s-1} \right)$;
this is justified when $\Re(s)$ is large enough and then extends by
analytic continuation. After some
rearrangement, one finds
\begin{equation} \label{4.7}
\frac{d}{du} \log T_W(\rho) = \frac12 \frac{d}{ds} \Bigg|_{s=0} \sum_{q=0}^N (-1)^{q+1} s \Tr
\left( \alpha (- \triangle_q)^{-s}  \right).
\end{equation}
The key term is the $(- \triangle_q)^{-s}$ appearing on the
right-hand side of (\ref{4.7}). Theorem \ref{3.12}, or more precisely its
extension to differential forms, implies that the expression
\begin{equation} \label{4.8}
\Tr
\left( \alpha (- \triangle_q)^{-s}  \right)
= \int_M \tr \left( \alpha(x) \: \zeta_{x,x}(s) \right) \: \dvol(x)
\end{equation}
vanishes at $s=0$.
Then the additional factor of $s$ in (\ref{4.7}) gives 
$\frac{d}{du} \log T_W(\rho) = 0$.
    \end{proof}

    Ray and Singer gave evidence that the analytic torsion equals the $R$-torsion.
    For example, they showed that the analog of (\ref{2.7}) holds for the analytic torsion.

    In some ways, it is convenient to remove the acyclicity assumption that
    $\HH^*(W, E) = 0$. If $\HH^*(W, E)$ is nonzero then
    Ray and Singer defined $T_W(\rho)$ as in (\ref{4.5}), where the
    zero eigenvalues of $\triangle_q$ are removed when constructing $\zeta_{q,\rho}(s)$.
    That is,
    \begin{equation} \label{4.9}
    \zeta_{q,\rho}(s) = \frac{1}{\Gamma(s)} \int_0^\infty t^{s-1} \Tr \left(
    e^{t \triangle_q} - P_{\Ker(\triangle_q)} \right) \: dt,
    \end{equation}
    where $P_{\Ker(\triangle_q)}$ denotes orthogonal projection onto the
    (finite dimensional) kernel of $\triangle_q$.
    On the $R$-torsion side, the Hodge isomorphism gives an $L^2$-inner product
    on $\HH^q(W, E)$, and hence an orthonormal basis.  Fixing this inner product,
    Ray and Singer showed that the $R$-torsion $\tau_{K}({\rho})$ is unchanged upon
    a subdivision of a triangulation $K$ of $W$. If the Riemannian metric varies,
    they computed the variations of $T_W(\rho)$ and $\tau_{K}({\rho})$. The variations
    were generally nonzero but just involved the change of volume form on 
    $\HH^*(W, E)$. In particular, they derived that $T_W(\rho)/\tau_{K}({\rho})$
    is metric independent, even in the nonacyclic case.

    Much of the Ray-Singer paper dealt with the technically more challenging case when
    $W$ is allowed to have boundary.  Without entering into the details, suffice it to
    say that most of their results for the closed case extend to the case of nonempty
    boundary.

\subsection{The second Ray-Singer paper} \label{sect4.2}

In a sequel, Ray and Singer defined an
analytic torsion in the holomorphic setting
\cite{Ray-Singer (1973)}.  In this case there was no classical counterpart.

The analogy was that a smooth manifold goes to a complex manifold, a
Riemannian metric goes to a Hermitian metric and the Hodge Laplacian
goes to the $\overline{\partial}$-Laplacian.

Suppose now that $W$ is a compact connected complex manifold of complex dimension $N$.
Let $\rho : \pi_1(W, w_0) \rightarrow U(n)$ be a homomorphism, with corresponding
flat holomorphic vector bundle $L$ on $W$.
For each $p$, there is a complex 
\begin{equation} \label{4.10}
    \ldots 
 \stackrel{\overline{\partial}}{\longrightarrow} \Omega^{p,q}(W, L) 
\stackrel{\overline{\partial}}{\longrightarrow}  
\Omega^{p,q+1}(W, L) \stackrel{\overline{\partial}}{\longrightarrow} \ldots
\end{equation}
and corresponding Laplacian $\triangle_p = - \left( \overline{\partial}^* \overline{\partial} 
+ \overline{\partial} \overline{\partial}^* \right)$. Let $\triangle_{p,q}$ denote the
restriction of $\triangle_p$ to $\Omega^{p,q}(W, L)$.
Put $\zeta_{p,q,\rho}(s) = \Tr \left( (- \triangle^\prime_{p,q})^{-s} \right)$, where
the $\prime$ on $\triangle^\prime_{p,q}$ indicates that zero eigenvalues are neglected.

\begin{definition} \label{4.11}
Given $p$, the $\overline{\partial}$-torsion $T_p(W, \rho)$
 is the positive real number
 such that
\begin{equation} \label{4.12}
\log T_p(W,\rho) = \frac12 \sum_{q=0}^N (-1)^q q \zeta^\prime_{p,q,\rho}(0).
    \end{equation}
\end{definition}

Regarding the dependence of $T_p(W,\rho)$ on the Hermitian metric, the
proof of Theorem \ref{4.6} goes through except for the last step.
Because the real dimension of $W$ is even, $\Tr
\left( \alpha (- \triangle_q)^{-s}  \right)$ need not vanish at $s=0$.
However, Ray and Singer deduced the following statement.
\begin{theorem} \label{4.13} \cite{Ray-Singer (1973)}
Let $\rho_1$ and $\rho_2$ be two homomorphisms from $\pi_1(W, w_0)$ to
$U(n)$, with corresponding flat vector bundles $L_1$ and $L_2$,
respectively. 
Given $p$, suppose that $\HH^{p,q}(W, L_j) = 0$ for all $q \in [0, N]$
and for $j \in \{1,2\}$.  Then $T_p(W,\rho_1)/T_p(W,\rho_2)$ is independent
of the Hermitian metric on $W$.
\end{theorem}

The reason that the theorem holds is that if $h(u)$ is a one-parameter
family of Hermitian metrics on $W$ then the
argument for Theorem \ref{4.6} shows that $\frac{d}{du} \log T_p(W,\rho_j)$
is an integral over $W$
whose integrand just depends on the local geometry.  In taking the
difference $\frac{d}{du} \left( \log T_p(W,\rho_1) -\log T_p(W,\rho_2)
\right)$ the integrand cancels out. If $N=1$ then one can remove the
assumption that $\HH^{p,q}(W, L_j) = 0$.

Hence under the assumptions of Theorem \ref{4.13}, the ratio
$T_p(W,\rho_1)/T_p(W,\rho_2)$ is an invariant of the complex manifold $W$.
Ray and Singer computed it explicitly for Riemann surfaces.  In the case of
genus one, i.e. tori, the answer was in terms of theta functions.
In the case of genus greater than one, they used the Selberg trace
formula to express the answer in terms of Selberg zeta functions.

\section{Further developments} \label{sect5}

After the first Ray-Singer paper, an outstanding problem was to show that the
analytic torsion equals the $R$-torsion.  Section \ref{sect5.1} describes the
proofs of this by Cheeger and M\"uller, along with
the subsequent proof by Bismut and Zhang.

A further understanding of the Ray-Singer torsion came from looking at {\em families}.  
Section \ref{sect5.2} explains how Quillen used the $\overline{\partial}$-torsion of the second
Ray-Singer paper in the setting of a family of $\overline{\partial}$-operators
on a complex vector bundle over a Riemann surface.
Section \ref{sect5.3} has the extension to higher dimension,
due to Bismut-Gillet-Soul\'e, along with their construction of a holomorphic
torsion form.  Section \ref{sect5.4} describes the analytic torsion form of a smooth
fiber bundle, due to Bismut and me.

\subsection{Cheeger-M\"uller theorem} \label{sect5.1}

In their first paper, Ray and Singer showed that the $R$-torsion 
$\tau_K(\rho)$ and the analytic torsion $T_W(\rho)$ have formal similarities.
Furthermore, Ray showed by explicit computation that they coincide for lens spaces
\cite{Ray (1970)}. This naturally lead to the problem of showing that
$T_W(\rho) = \tau_K(\rho)$ when
$K$ is a triangulation of $W$.  There are various proofs of this,
each being technically involved.

In order to show that $T_W(\rho)$ and $\tau_K(\rho)$ are the same, 
a first approach might be to use the
similarity between (\ref{2.8}) and (\ref{4.11}), take finer and finer triangulations
of $W$, and hope that (\ref{2.8}) approaches (\ref{4.11}).  This approach cannot
work directly since (\ref{4.11}) requires an analytic continuation beyond
singularities in the $s$-plane, whereas 
one does not see any such singularities in (\ref{2.8}).

Rather than trying to show directly that $T_W(\rho)/\tau_K(\rho) = 1$
Ray and Singer proposed to first take two representations
$\rho_1, \rho_2 : \pi_1(W, w_0) \rightarrow O(n)$ and show that
the difference
\begin{equation} \label{5.1}
\log \left( \frac{T_W(\rho_1)}{\tau_K(\rho_1)} \right) - \log 
\left( \frac{T_W(\rho_2)}{\tau_K(\rho_2)} \right)
=\log \left( \frac{T_W(\rho_1)}{T_W(\rho_2)} \right) - 
\log \left( \frac{\tau_K(\rho_1)}{\tau_K(\rho_2)} \right)
\end{equation}
vanishes.  Since the residues of the (simple) poles of the zeta function are
the integrals of local expressions on $W$, the difference
$\zeta_{q,\rho_1}(s)-\zeta_{q,\rho_2}(s)$ is analytic in $s$.
Because of this,
$\log (T_W(\rho_1)/T_W(\rho_2))$ should have better approximation properties
than $\log T_W(\rho_1)$ or $\log T_W(\rho_2)$
individually.

Ray and Singer's idea was to put a Morse function $f$ on $W$, look at
its sublevel sets $W_u = f^{-1}(-\infty, u]$, compute
the various torsions on $W_u$ (with appropriate boundary conditions)
and analyze how the
expression in (\ref{5.1}) depends on $u$.  Away
from the critical values of $f$, it should be constant in $u$. One would then
want to analyze how it changes when one passes through a critical value.
In this approach it is important to remove the acyclicity condition.

The equality of $T_W(\rho)$ and $\tau_K(\rho)$ was proven independently
by Jeff Cheeger \cite{Cheeger (1979)} and Werner M\"uller \cite{Muller (1978)}.
As in the Ray-Singer idea, the proofs involved analyzing how $T_W(\rho)/\tau_K(\rho)$ 
varies under topological change. How this was implemented differed from
what Ray and Singer had in mind.  

One common idea to the Cheeger and M\"uller papers was to use the fact that the torsions are 
defined independent of orientation, and the torsion of $W \cup W$ is twice the torsion of $W$.
As
$W \cup W$ is the boundary of $[0,1] \times W$, there is a sequence of surgeries that transform
$W \cup W$ to $S^N$.  If one can keep track of how the torsions change under surgery then
one can reduce to checking the equality on $S^N$.

M\"uller's first step was to use approximations of the differential form Laplacian by the combinatorial
Laplacian to show that (\ref{5.1}) vanishes.  Hence $T_W(\rho)/\tau_K(\rho)$
was independent of the representation and it sufficed to work with the
trivial representation $\rho$.  Next, suppose that one has a presurgery manifold $W_1$ and
a postsurgery manifold $W_2$. The surgery amounts to removing a copy of
$S^k \times D^{N-k}$ and adding a copy of $D^{k+1} \times S^{N-k-1}$. 
One can assume that the Riemannian metric is standard on those pieces. Let
$W_3$ be the double $S^k \times S^{N-k}$ of $S^k \times D^{N-k}$, and let
$W_4$ be the double $S^{k+1} \times S^{N-k-1}$ of $D^{k+1} \times S^{N-k-1}$.
In the combination $\zeta^{W_1}_{q,\rho}(s)-\zeta^{W_2}_{q,\rho}(s) 
- \frac12 \zeta^{W_3}_{q,\rho}(s) + \frac12 \zeta^{W_4}_{q,\rho}(s)$ the singularities cancelled
out, and so one obtained an analytic function of $s$.  M\"uller showed that this
combination can be approximated by the analogous expression in combinatorial
Laplacians, obtaining in the end that $T_{W_1}(\rho)/\tau_{K_1}(\rho) = 
T_{W_2}(\rho)/\tau_{K_2}(\rho)$.  This finally reduced to checking the ratio
for $W = S^N$, where it followed from Ray's calculations.

Cheeger's approach to the surgery was to keep careful track of the eigenvalues
under a conical degeneration.  Let $W_1(u)$ denote the result of 
removing $S^k \times D^{N-k}(u)$ from $W_1$, where $D^{N-k}(u)$ denotes
a disk of radius $u$.  Cheeger imposed absolute boundary conditions on
$W_1(u)$ and made a refined analysis of how the heat kernel for the
differential form Laplacian behaved as
$u \rightarrow 0$. A model space for this analysis was the product $S^k \times
A_{u,1}^{N-k}$, where $A_{u,1}^{N-k}$ denotes the annulus in $\R^{N-k}$ with outer radius one
and inner radius $u$.  With the symmetric analysis for $W_2$,
he was able to show that $T_{W_1}(\rho)/\tau_{K_1}(\rho) = 
T_{W_2}(\rho)/\tau_{K_2}(\rho)$, provided that $1 \le k \le n-2$.  
Finally, replacing $W$ by $W \times S^6$, he was
able to prove that $T_{W}(\rho)/\tau_{K}(\rho) = 1$. Besides proving the equality
of the torsions, Cheeger's methods lead to his later work on 
the spectral geometry of singular Riemannian spaces
\cite{Cheeger (1983)}.

Jean-Michel Bismut and Weiping Zhang gave an alternative proof of the Cheeger-M\"uller theorem
using a Morse function $f$ on $W$ \cite{Bismut-Zhang (1992)}. 
The idea was to do a Witten deformation, meaning that one replaces
$d$ by  $e^{-Tf} \circ d
\circ e^{Tf}$ and $\delta$ by $e^{Tf} \circ \delta
\circ e^{-Tf}$. As in the proof of Theorem \ref{4.6}, the ensuing
analytic torsion is independent of $T$. 

As $T \rightarrow \infty$, most of the eigenvalues of
$\triangle$ go to minus infinity.  The ones that stay bounded have
eigenfunctions that give
the (finite dimensional) Witten complex computing $\HH^*(W, E)$.
By means of this limit, using generic Morse functions $f$ (i.e.
$\nabla f$ satisfies the Smale transversality conditions) Bismut
and Zhang were able to prove that $T_W(\rho) = \tau_K(\rho)$
without performing surgery on $W$.

Finally, an approach using a gluing formula for the analytic torsion
was given by Simeon Vishik \cite{Vishik (1995)}.

\subsection{Determinant line bundle} \label{sect5.2}

In 1985, Daniel Quillen published a four page paper that gave a new
understanding of the $\overline{\partial}$-torsion \cite{Quillen (1985)}.
Quillen's paper had one reference,
the second Ray-Singer paper.  

As a first step, Quillen applied the definition of the $\overline{\partial}$-torsion 
not just for a flat unitary bundle, but rather for a general holomorphic
bundle equipped with a Hermitian inner product.  Since the analytic torsion
of the first Ray-Singer paper was defined using a $d$-flat vector bundle,
it is natural in the holomorphic setting to replace this by a 
$\overline{\partial}$-flat vector bundle, i.e. a holomorphic vector bundle $E$.
We write the corresponding $\overline{\partial}$-torsion as $T_W(E)$,
taking $p=0$.

Quillen considered a compact Riemann surface $W$ and a 
smooth complex vector bundle $E$ on $W$. A holomorphic structure on $E$ corresponds
to a choice of $\overline{\partial}$-operator
$\overline{\partial} : \Omega^{0,0}(W, E) \rightarrow \Omega^{0,1}(W, E)$;
the local holomorphic sections $s$ of $E$ correspond to the local solutions
of $\overline{\partial} s = 0$. 

Rather than looking at a single holomorphic structure on $E$, Quillen looked
at the family ${\mathcal A}$ of {\em all} such structures. It has a natural
complex structure.
In this setting, there is a holomorphic line bundle $Det$
on ${\mathcal A}$ called
the {\em determinant line bundle}.  
Given a holomorphic structure $a \in {\mathcal A}$ on $E$,
the fiber of $Det$ over $a$ is 
\begin{equation} \label{5.2}
Det_a = \left( \Lambda^{max} \HH^0(W, E) \right)^* \otimes 
\Lambda^{max} \HH^1(W, E),
\end{equation}
where $\Lambda^{max}$ denotes the highest
exterior power.  

Suppose that we want to put an inner product on $Det$.
Given a Riemannian metric on $W$ and a Hermitian inner product on $E$, for 
each $a \in {\mathcal A}$ the Hodge isomorphism gives
an $L^2$ inner product $\langle \cdot, \cdot \rangle_{L^2,a}$ on $Det_a$.  Unfortunately,  
this inner product need not be continuous in $a$.  The issue is that while $Det_a$
is smooth in $a$, the individual factors $\HH^0(W, E)$ and $\HH^1(W, E)$ can
abruptly jump in dimension.

Quillen's insight was that this lack of continuity can be corrected using the 
$\overline{\partial}$-torsion.  The {\em Quillen metric} on $Det_a$ is defined by
$\langle \cdot, \cdot \rangle_{Q,a} = T_W(E)^2
\langle \cdot, \cdot \rangle_{L^2,a}$. It gives rise to a {\em smooth} inner product on $Det$.

An inner product on a holomorphic line bundle induces a preferred connection on
the line bundle. One can then talk about the curvature of the connection.  In the case of the determinant line bundle,
computing the curvature essentially amounts to computing 
$\partial \overline{\partial} \log T_W(E)$,
a $2$-form on ${\mathcal A}$.
Quillen found a formula for the curvature involving the integral of a {\em local} expression on $W$,
in contrast to the nonlocal nature of $T_W(E)$,

If $\HH^0(W, E)$ and $\HH^1(W, E)$ both happen to vanish 
then the {\em determinant line} $Det_a$ can be
identified with $\C$, with a canonical element $1 \in Det_a$.  In this setting,
the Quillen norm of $1$ is $T_W(E)$. 
Although it may look as if we haven't achieved anything new,
the framework of determinant
line bundles is useful as one can sometimes use holomorphic methods to
{\em compute} $T_W(E)$ \cite{Bost-Nelson (1986)}.

\subsection{Holomorphic torsion form} \label{sect5.3}

Quillen's work was extended to higher dimensions by Jean-Michel Bismut, Henri Gillet and
Christophe Soul\'e \cite{BGS (1988)}. Furthermore, they found that the Ray-Singer torsion 
$T_W(E)$ entered into a differential form
version of the Riemann-Roch-Grothendieck (RRG) theorem.

The setup of \cite{BGS (1988)} was a family of complex structures on a compact manifold
$Z$, parametrized by a complex manifold $B$.  That is, one has a holomorphic fiber bundle
$\pi : M \rightarrow B$ whose fibers are diffeomorphic to $Z$. They also assumed that 
the fibers carry K\"ahler metrics that form a K\"ahler fibration, in the sense that the
K\"ahler forms on the fibers are the restrictions of a closed $(1,1)$-form on $M$.

Let $E$ be a holomorphic vector bundle on $M$, equipped with a Hermitian inner product $h^E$. 
There is an induced Chern
connection on $E$.
Put $Z_b = \pi^{-1}(b)$. 
The determinant line bundle $Det$ is a holomorphic line bundle on $B$ whose fiber
over $b \in B$ is 
\begin{align} \label{5.3}
Det_b = & \left( \Lambda^{max} \HH^0 \left( Z_b, E \big|_{Z_b} \right) \right)^* \otimes 
\Lambda^{max} \HH^1 \left( Z_b, E \big|_{Z_b} \right) \otimes \\
& \left( \Lambda^{max} 
\HH^2 \left( Z_b, E \big|_{Z_b} \right)
\right)^* \otimes 
\Lambda^{max} \HH^3 \left( Z_b, E \big|_{Z_b} \right) \otimes \ldots  \notag
\end{align}
If $\langle \cdot, \cdot \rangle_{L^2,b}$ denotes the $L^2$-inner product on
$Det_b$ then 
the {Quillen metric} $\langle \cdot, \cdot \rangle_{Q}$ on $Det$ is defined by
$\langle \cdot, \cdot \rangle_{Q,b} = T\left( Z_b, E \big|_{Z_b} \right)^2
\langle \cdot, \cdot \rangle_{L^2,b}$.

Put $TZ = \Ker(d\pi)$, a holomorphic vector bundle on $M$.  The K\"ahler fibration gives a connection on $TZ$.

\begin{theorem} \label{5.4} \cite{BGS (1988)} The curvature $2$-form associated to the Quillen metric is
$2 \pi i$ times the $2$-form component of $\int_Z \Td(TZ, g^{TZ}) \wedge \ch(E, h^E)$.
\end{theorem}
Here  
$\int_Z$ is integration over the fiber, 
$\Td$ is the Todd form and
$\ch$ is the Chern character form. The normalization is such that
$\Td$ and $\ch$ represent rational cohomology classes.
The validity
of the theorem was previously known on the level of cohomology.
The point is that an
explicit $2$-form representative arises geometrically
as the curvature of $Det$, when the latter is equipped with the Quillen metric.

On the level of cohomology, the expression $\int_Z \Td(TZ) \: \cup \: \ch(E)$ is the
right-hand side of the RRG theorem. To state the theorem, let us
make the additional assumption that for each $q$, the dimension of 
$\HH^q \left( Z_b, E \big|_{Z_b} \right)$ is constant in $b$. Then the
vector spaces $\left\{ \HH^q \left( Z_b, E \big|_{Z_b} \right) \right\}_{b \in B}$ fit together
to form a holomorphic
vector bundle $\underline{\HH}^q$ on $B$. In this setting the RRG
theorem says that 
\begin{equation} \label{5.5}
\sum_{q=0}^{\dim_\C (Z)} (-1)^q \ch(\underline{\HH}^q) = \int_Z \Td(TZ) \cup \ch(E),
\end{equation}
in $\HH^{even}(B; \R)$.

The vector bundle $\underline{\HH}^q$ acquires an $L^2$-inner product 
$h^{\underline{\HH}^q}$ and
corresponding connection. One can ask whether (\ref{5.5}) becomes an equality of
differential forms.  This is not the case, but the discrepancy can be
described using the {\em holomorphic torsion form}.

\begin{theorem} \label{5.6} \cite{BGS (1988)} There is a canonical form
${\mathcal T} \in \Omega^{even}(B)$, depending on the above geometric data, so that
\begin{equation} \label{5.7}
\sqrt{-1} \partial \overline{\partial} {\mathcal T} =
\int_Z \Td \left( TZ, \nabla^{TZ} \right) \wedge \ch(E, h^E) -
\sum_{q=0}^{\dim_\C (Z)} (-1)^q \ch \left( \underline{\HH}^q, h^{\underline{\HH}^q} \right)
\end{equation}
in $\Omega^{even}(B)$.  The degree-zero component ${\mathcal T}_{[0]} \in C^\infty(B)$
of ${\mathcal T}$ is related to the $\overline{\partial}$-torsion by
${\mathcal T}_{[0]}(b) = \frac{1}{\pi} \log T\left( Z_b, E \big|_{Z_b} \right)$.
\end{theorem}

Because of Theorem \ref{5.6}, the form ${\mathcal T}$ can be called the {\em holomorphic
torsion form}.  It can be considered to be a {\em transgression} of the RRG theorem,
on the level of differential forms. 

As a remark, the formalism of determinant line bundles and Quillen metrics extends to
smooth families of Dirac-type operators 
\cite[Chapters 9.7 and 10.6]{Berline-Getzler-Vergne (2004)},
\cite{Bismut-Freed (1986),Bismut-Freed (1987)}.  The definition of the Quillen metric
again involves  a product of
regularized determinants, although in general it cannot be identified
with the Ray-Singer torsion.

\subsection{Analytic torsion form} \label{sect5.4}

As described in the previous section, the
$\overline{\partial}$-torsion is the
$0$-form component of a torsion form that
represents a transgression of the RRG
formula.  
To come full circle, one can ask if there's
a similar interpretation of the original
Ray-Singer torsion.  It turns out that there is, as was shown by
Bismut and me.

In order to see this interpretation, it was necessary to extend the
definition of the Ray-Singer torsion $T_W(\rho)$ beyond the case of
flat orthogonal or unitary vector bundles.  Let $E$ be a flat complex vector bundle
on $W$.  Suppose that $E$ is equipped with a Hermitian inner product
$h^E$, not necessarily covariantly constant.  Then one can still use
the formula (\ref{4.5}) to define the Ray-Singer torsion $T_W(E)$.

Considering the role that volume forms play in the $R$-torsion, a
natural extension of the Ray-Singer work was to assume that $E$ has
unimodular holonomy, i.e.
taking values in $\{ A \in \GL(n, \C) \: : \:
|\det A| = 1\}$. In this case M\"uller proved the extension of the
Cheeger-M\"uller theorem \cite{Muller (1993)}. This had later application
to the growth of torsion in the cohomology of locally symmetric
spaces, as described in \cite{Muller (2022)}.

Going beyond this, Bismut and Zhang considered arbitrary flat complex
vector bundles $E$ \cite{Bismut-Zhang (1992)}.  The topological 
meaning of $T_W(E)$ was not so clear in this case but Bismut and
Zhang proved ``anomaly'' formulas showing how $T_W(E)$ depends on $g_W$ and $h^E$.

Using the analogy that $\overline{\partial}$-flat vector bundles, i.e. holomorphic bundles, are like
$d$-flat vector bundles, i.e. flat vector bundles,
the first question was whether there is a
analog of the RRG theorem for flat vector bundles. It turns out that there is.
To state it, let us define certain characteristic classes
of {\em flat} vector
bundles.  

Let $W$ be a smooth manifold and
let $E$ be a flat complex vector bundle on $W$. Let $h^E$ be a Hermitian inner product on 
$E$ (not necesarily covariantly constant).  With respect to a local covariantly 
constant basis of $E$, we can think of $h^E$ locally as a Hermitian matrix valued
function on $B$.  Put $\omega(E, h^E) = (h^E)^{-1} dh^E$, a globally defined
$\End(E)$-valued $1$-form on $B$. 

If $k$ is a positive odd integer, put 
\begin{equation} \label{5.8}
c_k(E, h^E) = (2 \pi i)^{- \frac{k-1}{2}} 2^{-k}
\tr \left( (\omega(E, h^E))^k \right).
\end{equation}
Then $c_k(E, h^E)$ is closed and its de Rham cohomology class
$c_k(E)$ is independent of $h^E$.  To give some idea of what
$c_k(E)$ measures, it vanishes if $E$ has unitary holonomy.
And $c_1(E)$ vanishes if the holonomy is unimodular.

Now let $\pi : M \rightarrow B$ be a fiber bundle with
closed fibers $Z_b = \pi^{-1}(b)$. Let $E$ be a flat complex
vector bundle on $M$ and let $\underline{\HH}^q$ be the flat
complex vector bundle on $B$ whose fiber over $b \in B$ is
$\HH^q \left( Z_b, E \Big|_{Z_b} \right)$. Let $TZ = \Ker(d\pi)$
be the vertical tangent bundle, a vector bundle on $M$, and let
$o(TZ)$ be its orientation bundle, a flat $\R$-bundle on $M$.
Let $e(TZ) \in \HH^{\dim(Z)}(M; o(TZ))$ be the Euler class of $TZ$.
The following is an analog of the RRG theorem, for flat vector
bundles.

\begin{theorem} \label{5.9} \cite{Bismut-Lott (1995)}
For any positive odd number $k$,
\begin{equation} \label{5.10}
\sum_{q=0}^{\dim (Z)} (-1)^q c_k(\underline{\HH}^q) = \int_Z e(TZ) \cup c_k(E)
\end{equation}
in $\HH^k(B; \R)$.
\end{theorem}

To explain where the torsion comes in,
equip the fiber bundle with a horizontal
distribution $T^HM$ and a vertical
Riemannian metric $g^{TZ}$. Also equip
$E$ with a Hermitian inner product.
Then the vector bundle $\underline{\HH}^q$
acquires an $L^2$-inner product $h^{\underline{\HH}^q}$. 
The next result states the existence of
``higher'' analytic torsion forms that realize
(\ref{5.10}) on the level of differential forms.

\begin{theorem} \label{5.11} \cite{Bismut-Lott (1995)}
For any positive odd number $k$,
there
is an explicit $(k-1)$-form
${\mathcal T}_{k-1}$, depending on the
geometric data, so that 
\begin{equation} \label{5.12}
d{\mathcal T}_{k-1} = \int_Z e \left( TZ, \nabla^{TZ} \right) \wedge c_k(E, h^E) - \sum_{q=0}^{\dim (Z)} (-1)^q c_k \left( \underline{\HH}^q, h^{\underline{\HH}^q} \right)
\end{equation}
in $\Omega^k(B; \R)$. When $k = 1$, the function ${\mathcal T}_0 \in C^\infty(B)$ is such that
${\mathcal T}_0(b)$ is the negative of the logarithm of the Ray-Singer torsion
$T_{Z_b}( E \Big|_{Z_b})$. 
\end{theorem}

When $k=1$, equation (\ref{5.12}) recovers the anomaly formula of Bismut and Zhang.

If $\dim(Z)$ is odd and $\underline{\HH}^q = 0$
for all $q$ then (\ref{5.12}) implies that
${\mathcal T}_{k-1}$ is closed, and hence
has a de Rham representative
$[{\mathcal T}_{k-1}] \in \HH^{k-1}(B, \R)$.
It turns out that $[{\mathcal T}_{k-1}]$
is independent of the choices of
$T^H M$, $g^TZ$ and $h^E$, i.e. just depends
on the smooth  fiber bundle $\pi : M \rightarrow B$ and the flat complex vector
bundle $E \rightarrow M$. On the other hand, there are ``higher'' versions of
the $R$-torsion, which are also invariants of smooth fiber bundles
\cite{Dwyer-Weiss-Williams (2003),Igusa (2002)}. The question then arises if there is a higher Cheeger-M\"uller
theorem.  The latest on this is
\cite{Puchol-Zhang-Zhu (2021)}.

\end{document}